\documentclass[smallextended,final]{svjour3arX}
\usepackage{amssymb,latexsym}
\usepackage{amsmath}
\usepackage{tikz}
  \usetikzlibrary{arrows,topaths}
\usepackage{verbatim}
\usetikzlibrary{decorations.pathreplacing,angles,quotes}
\usepackage{rawfonts}
\usepackage{amsmath, psfrag}
\usepackage{amsfonts}
\usepackage{lineno,hyperref}
\newcommand{\F}{\mathcal F}

\def\pont{\hspace{-5pt}{\bf.\ }}
\begin{document}
\title{\large The Szemer\'edi-Petruska conjecture for a few small values}

\author{\it\large Adam S. Jobson\; \\ Andr\'e E. K\'ezdy\; \\ Jen\H{o} Lehel}

\institute{Adam S. Jobson \at
              Department of Mathematics\\
              University of Louisville,  Louisville, KY 40292\\            
              \email{adam.jobson@louisville.edu}     
  \and
           Andr\'e E. K\'ezdy \at
           Department of Mathematics\\
           University of Louisville,  Louisville, KY 40292\\            
              \email{kezdy@louisville.edu}   
   \and
           Jen\H{o} Lehel \at                                 
              Alfr\'ed R\'enyi Mathematical Institute\\ 
              Budapest, Hungary\\
              and\\
              Department of Mathematics\\
           University of Louisville,  Louisville, KY 40292\\    
              \email{lehelj@renyi.hu}  
}

\date{}
\maketitle
\begin{abstract}\pont  Let $H$ be a $3$-uniform hypergraph of order $n$ with clique number $\omega(H)=k$ such that the intersection of all  maximum cliques of $H$  is empty. For fixed $m=n-k$, Szemer\'edi and Petruska conjectured the sharp bound $n\leq {m+2\choose 2}$. In this note the conjecture is verified for $m=2,3$ and $4$. 
\end{abstract}
\vspace{2em}
\section{}


 Let $H$ be a $3$-uniform hypergraph of order $n$ with clique number $\omega(H)=k$ such that
the intersection of all  maximum cliques of $H$  is empty. For fixed $m=n-k$ Szemer\'edi and Petruska \cite{SzP} conjectured the tight bound $n\leq {m+2\choose 2}$. The construction showing the tightness of the bound  is also conjectured to be unique. 
 
 The Szemer\'edi and Petruska conjecture is equivalent with the statement that ${m+2}\choose{2}$ is the maximum order of  a $3$-uniform $\tau$-critical hypergraph\footnote{A hypergraph is $\tau$-critical if it has no isolated vertex and the removal of every edge decreases its transversal number}, see  \cite[Problem 18.(a)]{Tu85} with transversal number $m$. For the maximum order Tuza\footnote{Personal communication}   obtained the best known bound  $\frac{3}{4}m^2 + m + 1$ using the
machinery of $\tau$-critical hypergraphs. 
 
 An alternative approach  is proposed by  Jobson et al. 
 \cite{JKP}, combining a
decomposition process introduced by Szemer\'edi and Petruska \cite{SzP} with the skew version
of Bollob\'as's theorem \cite{BB}, 
 and using tools from linear algebra.
 
It turns out that the Szemer\'edi-Petruska conjecture has applications in extremal problems concerning convex sets in the plane, see Jobson et al. \cite{eck} and \cite{petrus}. The validity of the Szemer\'edi-Petruska conjecture for small values carries relevant information pertaining to those combinatorial geometry problems.

\section{} 
We prove here the Szemer\'edi-Petruska conjecture for $m=2,3$, and $4$.

\begin{proposition}\pont 
\label{main} {
Let $m=2,3$, or $4$, and $n>m$. If $H$ is a $3$-uniform hypergraph of order $n$ with clique number $\omega(H)=n-m=k\geq 3$, and  the intersection  of 
the  $k$-cliques of $H$ 
 is empty, then $n\leq {m+2\choose 2}$.}
\end{proposition}

\noindent{\it Proof.} 
Let  ${\cal N} = \{N_1,\ldots,N_\ell\}$ ($\ell\geq 3$) be a collection of $k$-cliques of  $H=(V,E)$
 such that
 $\bigcap_{i=1}^\ell N_i = \varnothing$, but $\bigcap_{j\neq i} N_j \neq \varnothing$ for all $i=1,\ldots,\ell$. W.l.o.g. we assume that
 $\bigcup_{i=1}^\ell  N_j =V$, and $|V|=n$.
 Szemer\'edi and Petruska \cite[Lemma 4]{SzP} observed that
each $N_i$, $i=1,\ldots,\ell$, 
contains a {\it private pair} $p_i\subset V$, $|p_i|=2$, such that $p_i\subset N_j$ if and only if $i=j$.
Let $G=(V,E)$ be the graph of the private pairs $p_i=\{a_i,b_i\}$, $1\leq i\leq \ell$, as its edges. 

Set $M_i=V\setminus N_i$, and  notice that $|M_i|=n-k=m$.
By the definition of private pairs, $(p_i,M_i) $, $1\leq i\leq \ell$, form an intersecting
$(2,m)$-system:
$p_i\cap M_j=\varnothing $ if and only if $i=j$. Furthermore,  
 $\bigcup  p_i\subseteq \bigcup M_i=V$ (since $\bigcap N_i=\varnothing$).

  For an edge $p_j=\{a_j,b_j\}$ of  $ G$, define the {\it truncated subgraph} $G\setminus p_j$ on vertex set $V\setminus\{a_j,b_j\}$  by including
 $p_h\setminus \{a_j,b_j\}$ as an edge or a loop, for all
$h\neq j$. 
Since $M_j\cap p_h\neq \varnothing$ for all $h\neq j$,  we have 
\begin{equation}
\label{inter}
\tau(G\setminus p_j)\leq |M_j|=m,
\end{equation}
 for every $1\leq j\leq\ell$.
 For each edge $p_j\in G$ we define the {\it weight} 
$w(p_j)= m-\tau(G\setminus p_j)$. 
Let $G-p_j$ denote the graph obtained by the removal of edge $p_j$ from $G$ (and nothing else). Obviously we have
\begin{equation}
\label{subtau}
\tau(G)-1\leq\tau(G-p_j)\leq \tau(G\setminus p_j)\leq m.
\end{equation}
The conditions $V=\bigcup\limits_{i=1}^\ell M_i$, $\ell\geq 3$, and (\ref{subtau}) imply that $2\leq \tau(G)\leq m+1$. Furthermore,
$
0\leq w(p_j)=m-\tau(G\setminus p_j)\leq m-\tau(G)+1.
$

Let $V_0(G)=\bigcup\limits_{p\in E} p$, set  $Z(G)=V\setminus V_0(G)$ and $w(G)=\sum\limits_{p\in E}  w(p)$. 

\begin{lemma}\pont
\label{estime}
$\quad |V|= |V_0(G)|+|Z(G)|\leq |V_0(G)|+w(G).$
\end{lemma}
\begin{proof} \pont
Set $V_0=V_0(G)$, and  $Z_j=M_j\setminus V_0$. 
Observe that (\ref{inter}) implies
$|M_j\cap V_0|\geq \tau(G\setminus p_j)$, hence
$|Z_j|\leq (|M_j|- |M_j\cap V_0|)\leq m-\tau(G\setminus p_j)=w(p_j).$ In words, $M_j$ has at most $w(p_j)$ vertices not in $V_0$. 
Thus $|Z(G)| \leq \sum\limits_{j=1}^\ell |Z_j|\leq
\sum\limits_{j=1}^\ell  w(p_j)=w(G),$ and the claim follows.\qed
\end{proof}

\noindent {\it Step 1:} $\tau(G)=2$. 
Let $\{u,v\}$ be a transversal set of $G$, set $V_0=V_0(G)$, and 
let $X=\{a\in V_0\setminus\{v\} : ua\in E\}$ and $Y=\{a\in V_0\setminus\{u\} : va\in E\}$. Set
 $x=|X|$, 
 $y=|Y|$, and  $z=|X\cap Y|$.  

   Case A:  $Y\subseteq X$.
   In this case  $|V_0|=|X\cup\{u,v\}|=x+2$, and 
   
  If $q=uv\in E$, then  $w(q)=m-|X|=m-x$, otherwise set $w(q)=0$;

$\tau(G\setminus ua)=|X\setminus\{a\}|=x-1$, $w(ua)=m-x+1$ if  $a\in X\setminus Y$;

$\tau(G\setminus ua)=|(X\setminus\{a\})\cup\{v\}|=x$, $w(ua)=m-x$ if  $a\in Y$;

$\tau(G\setminus va)=|(Y\setminus\{a\})\cup\{u\}|=y$, $w(va)=m-y$ if  $a\in Y$. 

The bound in Lemma \ref{estime} becomes 
  \begin{eqnarray*}
|V|&\leq& |V_0|+w(G)=(x+2)+w(q)+\sum\limits_{p\neq q} w(p)\\
&\leq& (x+2)+(m-x)+(x-y)(m-x+1)+y(m-x)+y(m-y)\\
&=&m+2+ x(m-x+1) + y(m-y-1)\\
&\leq&m+2+
 \left\lfloor\frac{m+1}{2}\right\rfloor\cdot\left\lceil\frac{m+1}{2}\right\rceil+
 \left\lfloor\frac{m-1}{2}\right\rfloor\cdot\left\lceil\frac{m-1}{2}\right\rceil
\leq {m+2\choose 2}.
\end{eqnarray*}

  Case B: $X\setminus Y\neq \varnothing$ and $Y\setminus X\neq \varnothing$. 
     In this case  $|V_0|=x+y-z+2$, and if $q=uv\in E$,
 then  $\tau(G\setminus q)=x+y-z$, $w(q)=m-x-y+z$; furthermore,  
$\tau(G\setminus ua)=x$, $w(ua)=m-x$ if  $a\in X$, and
$\tau(G\setminus va)=y$, $w(va)=m-y$ if  $a\in Y$. 
Thus by applying Lemma \ref{estime} we obtain the bound 
\begin{eqnarray*}
|V|&\leq & (x+y-z+2)+(m-x-y+z)+x(m-x)+y(m-y)\\
&\leq& m+ 2+2 \left\lfloor\frac{m}{2}\right\rfloor\cdot\left\lceil\frac{m}{2}\right\rceil
\leq {m+2\choose 2}.
\end{eqnarray*}
 
Therefore, if  $\tau(G)=2$ then $n\leq {m+2\choose 2}$ follows for every 
$m\geq 2$.
\vspace{.5em}

\noindent {\it Step 2:} $\tau(G)=m+1$ or $m$.


Assume $\tau(G)=m+1$.  
By (\ref{subtau}),
$0\leq w(p_i)=m- \tau(G\setminus p_i)\leq m-\tau(G-p_i)$, for every $1\leq i\leq\ell$, thus we have $\tau(G-p_i)\leq m$. Therefore, $G$  is a $\tau$-critical graph. The vertex bound 
for $\tau$-critical graphs due to Erd\H{o}s and Gallai \cite{EG} implies $|V|\leq 2(m+1)$. For $m\geq 2$, $|V|\leq 2(m+1)\leq {m+2\choose 2}$ follows. 

Assume now that $\tau(G)=m$.   Notice that $w(p_i)=0$ means $M_i\subset V_0(G)$. 
By successively removing  edges from $G$ 
with $w(p_i)=0$ 
we conclude with  a subgraph $G^*$ such that $\tau(G^*)=m$ and $0<w(p)=m-\tau(G^*\setminus p)\leq m-\tau(G^*- p)$, for every $p\in E^*$. Hence $G^*$ is $\tau$-critical graph. We claim that  
 $|V|\leq |V^*|+|E^*|$, and then $|V|\leq {m+2\choose 2}$ follows from the combined bound 
$|V^*|+|E^*|\leq {m+2\choose 2}$ due to Gy\'arf\'as and Lehel \cite{GYL}. The claim follows by repeatedly applying the next lemma.

\begin{lemma}\pont
\label{w=0} 
If $\ell\geq 3$ and $w(p_\ell)=0$, then 
$G^\prime=G-p_\ell$ satisfies
\begin{equation}
\label{increase}
 |V_0(G)|+w(G)\leq |V_0(G^\prime)|+w(G^\prime).
\end{equation}
\end{lemma}
\begin{proof} \pont
Let $G^\prime$. For $1\leq i\leq \ell-1$, define $w^\prime(p_i)=m-\tau(G^\prime\setminus p_i)$.  Since a transversal set in
$G\setminus p_i$
is also a transversal set in $G^\prime\setminus p_i$, we have 
$\tau(G\setminus p_i)\geq \tau(G^\prime\setminus p_i)$. Then we obtain
$$w(G^\prime)=\sum\limits_{j=1}^{\ell-1} w^\prime(p_j)=\sum\limits_{j=1}^{\ell-1} (m-\tau(G^\prime\setminus p_i))\geq \sum\limits_{j=1}^{\ell-1} (m-\tau(G\setminus p_i))=w(G).
$$
If $p_\ell$ is neither an isolated edge  nor a pendant edge, then 
$V_0(G^\prime)=V_0(G)$, therefore
$ |V_0(G)|+w(G)\leq |V_0(G^\prime)|+w(G^\prime)$ follows.

If $p_\ell$ is an isolated edge in $G$,
then $|V_0(G^\prime)|=|V_0(G)|-2$.
Because $w(p_\ell)=0$, we have
$\tau(G)=\tau(G\setminus p_\ell)+1=m-w(p_\ell)+1=m+1$. Moreover, $G$ is $\tau$-critical, since by (\ref{subtau}), 
$$\tau(G)-1\leq \tau(G-p_j)\leq m=\tau(G)-1,$$
for every $1\leq j\leq\ell$. Then $w(p_j)=m-\tau(G\setminus p_j)=0$ thus $w(G)=0$. Observe that $G^\prime$ is $\tau$-critical, as well, with $\tau(G^\prime)=\tau(G)-1=m$, therefore, $w^\prime(p_j)=m-\tau(G^\prime\setminus p_j) =m-\tau(G^\prime- p_j)=1$. Then by 
the definition of the edge waights, we have
$w(G^\prime)=\sum\limits_{j=1}^{\ell-1} 1=|E^\prime|.$
Since $|E^\prime|\geq m\geq 2$, we obtain  
 $ |V_0(G)|+w(G)$ $=|V_0(G)|\leq (|V_0(G)|-2) + |E^\prime|$ $
= |V_0(G^\prime)|+w(G^\prime).$

If $p_\ell$ is a pendant edge, then $|V_0(G^\prime)|=|V_0(G)|-1$.
For any $p_i\cap p_\ell\neq\varnothing$, 
$\tau(G^\prime\setminus p_i)=\tau(G\setminus p_i)-1$, because a minimum transversal in $G\setminus p_i$ must use a vertex just for $p_\ell$, and this vertex is not required in  a minimum transversal of
$G^\prime\setminus p_i$. Hence $w^\prime(p_i)=w(p_i)+1$, which implies 
$w(G^\prime)\geq w(G)+1$, and 
 $ |V_0(G)|+w(G)$ $=(|V_0(G)|-1) +(w(G)+1)$ $\leq 
|V_0(G^\prime)|+w(G^\prime)$ 
 follows. \qed
 \end{proof}
 
 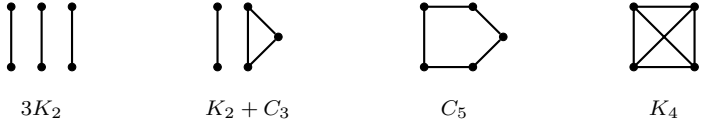
\begin{figure}[http]
 \tikzstyle{A} = [circle, draw=black!, minimum width=1pt, fill=white,inner sep=1pt]
\tikzstyle{P} = [circle, draw=black!, minimum width=0pt, inner sep=1pt, fill=black]
\begin{center}
\begin{tikzpicture}[scale=.8]
\foreach \x in{1,1.5,2}\foreach \y in{.5,1.5}\node[P] ()at(\x,\y){};
\foreach \x in{1,1.5,2}\draw[line width=.8pt](\x,.5)--(\x,1.5);\node()at(1.5,-.2){$3K_2$};
\end{tikzpicture}
\hskip1.5cm
\begin{tikzpicture}[scale=.8]

\foreach \x in{1,1.5}\foreach \y in{.5,1.5}\node[P] ()at(\x,\y){};
\foreach \x in{1,1.5}\draw[line width=.8pt](\x,.5)--(\x,1.5);
\node[P] (c)at(2,1){};
\draw[line width=.8pt](1.5,.5)--(c)--(1.5,1.5);
\node()at(1.5,-.2){$K_2+C_3$};
\end{tikzpicture}
\hskip1.5cm
\begin{tikzpicture}[scale=.8]

\foreach \x in{1,1.8}\foreach \y in{.5,1.5}\node[P] ()at(\x,\y){};
\foreach \x in{1}\draw[line width=.8pt](\x,.5)--(\x,1.5);
\draw[line width=.8pt](1,1.5)--(1.8,1.5) (1,.5)--(1.8,.5);
\node[P] (c)at(2.3,1){};
\draw[line width=.8pt](1.8,.5)--(c)--(1.8,1.5);
\node()at(1.5,-.2){$C_5$};
\end{tikzpicture}
\hskip1.5cm
\begin{tikzpicture}[scale=.8]

\foreach \x in{1,2}\foreach \y in{.5,1.5}\node[P] ()at(\x,\y){};
\foreach \x in{1,2}\draw[line width=.8pt](\x,.5)--(\x,1.5);
\draw[line width=.8pt](1,1.5)--(2,.5) --(1,.5)(1,1.5)--(2,1.5)--(1,.5);
\node()at(1.5,-.2){$K_4$};
\end{tikzpicture}

\caption{$\tau$-critical graphs $G$ with $\tau(G)=3$ }
\label{tau23}
\end{center}
\end{figure}

Because $2\leq\tau(G)\leq m+1$, Steps 1 and 2 imply 
$n\leq {m+2\choose 2}$, for $m=2,3$.
The case
 $m=4$, $\tau(G)=3$ remains to discuss.
By Lemma \ref{w=0} we may assume that all edges of $G$ have positive weights. Since $\tau(G)=3$, 
$G$ contains one of the four $\tau$-critical  subgraphs in Fig.\ref{tau23}.\vspace{.5em}

Case 1: $G$ cotains a $K_4$. If $G$ had two more non-isolated vertices, $a$ and $b$, then $ab \notin E$, since $\tau(G)=3$. One edge from each of $a$ and $b$ to $K_4$ would result in an edge of zero weight. Hence $|V_0|=5$ or $4$, and $G$ can be obtained starting with 
a $K_5$, and successively removing edges incident with a common vertex, say $a$. The corresponding configurations
are depicted in Fig.\ref{4clique} (unlabeled edges have  weight $1$).

\begin{figure}[htp]
 \begin{center}
 \tikzstyle{A} = [circle, draw=black!, minimum width=1pt, fill=white,inner sep=1pt]
\tikzstyle{P} = [circle, draw=black!, minimum width=0pt, inner sep=1pt, fill=black]
\begin{tikzpicture}
\node[P] (a)at(1,.5){};\node[P] (d)at(1.05,1.15){};
\node[P] (b)at(.2,1.4){};\node[P] (c)at(1.8,1.4){};
\draw[line width=.8pt](a)--(c) --(b)--(d)--(a)--(b) (c)--(d);
\node[P,label=above:$a$] (e)at(1.5,2.2){};
\draw[line width=.8pt](d)--(e)--(b) (e)--(c);
\node()at(1,0){$n\leq14$};
\end{tikzpicture}\hskip1cm
\begin{tikzpicture}
\node[P] (a)at(1,.5){};\node[P] (d)at(1.05,1.15){};
\node[P] (b)at(.2,1.4){};\node[P] (c)at(1.8,1.4){};
\draw[line width=.8pt](a)--(c) --(b)--(d)--(a)--(b) (c)--(d);
\node[P,label=above:$a$] (e)at(1.5,2.2){};
\draw[line width=.8pt](d)--(e)--(b);
\node()at(1,0){$n\leq13$};
\end{tikzpicture}\hskip1cm
\begin{tikzpicture}[scale=1.4]

\node[P,label=below:$$] (a)at(1,.5){};
\node[P,label=above:$$] (d)at(1.05,1.15){};
\node[P,label=above:$$] (b)at(.2,1.6){};
\node[P,label=above:$$] (c)at(1.8,1.6){};

\draw[line width=.8pt](a)--(c) --(b)--(d)--(a)--(b) (c)--(d);
\node()at(1.5,.95){\small{2}};
\node()at(1.2,1.4){\small{2}};
\node()at(.9,.95){\small{2}};
\node[P,label=right:$a$] (e)at(1.3,2.2){};
\draw[line width=.8pt](e)--(b);
\node()at(1,0){$n\leq15$};
\end{tikzpicture}\hskip1cm
\begin{tikzpicture}[scale=1.4]
\node[P,label=below:$d$] (d)at(1,.5){};
\node[P,label=above:$e$] (e)at(1.05,1.15){};
\node[P,label=above:$b$] (b)at(.2,1.6){};
\node[P,label=above:$c$] (c)at(1.8,1.6){};
\draw[line width=.8pt](a)--(c) --(b)--(e)--(a)--(b) (c)--(e);
\node()at(0.6,.8){\small{2}};\node()at(1.4,.8){\small{2}};
\node()at(0.7,1.2){\small{2}};\node()at(1.33,1.2){\small{2}};
\node()at(.9,.9){\small{2}};\node()at(.95,1.73){\small{2}};
\node()at(1,0){$n\leq16$};
\end{tikzpicture}
\caption{$G$ has a $K_4$}
\label{4clique}
\end{center}
\end{figure}
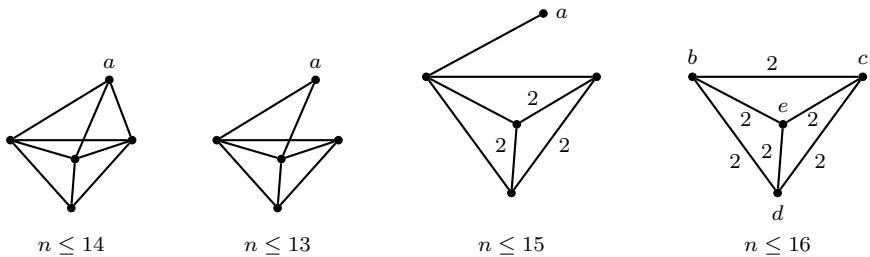

The  rightmost candidate in
 Fig.\ref{4clique} can be rejected as follows.
If $n= 16>{4+2\choose 2}$, then  
$Z_i=M_i\setminus V_0$, $1\leq i\leq 6$, are  pairwise disjoint $2$-element sets.
Let $V_0=\{b,c,d,e\}$ and $Z_i=\{i,i^\prime\}$. The corresponding $(2,m)$-system is uniquely determined as follows:
$ p_1=bc$, $p_2=bd, p_3=be,  p_4=cd,p_5=ce,
p_6=de$, and 
${M_1}=\{d,e,1,1^\prime\},$
${M_2}=\{c,e,2,2^\prime\},$
${M_3}=\{c,d,3,3^\prime\}, $
${M_4}=\{b,e,4,4^\prime\},$ 
${M_5}=\{b,d,5,5^\prime\},$
${M_6}=\{b,c,6,6^\prime\}. $
 
We eliminate the `fake' candidate by applying a general observation, the {\it Triples test}, as follows. If a set $f\subset V$, $|f|=3$, is such that $f\cap M_i=\varnothing$ for some $1\leq i\leq \ell$, then $f\subset N_i$ and since $N_i$ is a clique of $H$, we have $f\in E$.

Let $N=V\setminus\{c,d,e\}$;  $|N|=n-3=n-(m-1)=k+1$.  Notice that the four sets $M_i\setminus\{c,d,e\}$,
$i=1,2,3,4$, are pairwise disjoint, thus 
every $3$-element set $f\subset N$ is disjoint from some $M_i$, $1\leq i\leq 6$. Hence, by the Triples test,   $f\in E$ for every $f\subset N$, $|f|=3$. 
Therefore, $N$ induces a clique of order $k+1$, a contradiction. \vspace{.5em}

\begin{figure}[htp]
 \begin{center}
 \tikzstyle{A} = [circle, draw=black!, minimum width=1pt, fill=white,inner sep=1pt]
\tikzstyle{P} = [circle, draw=black!, minimum width=0pt, inner sep=1pt, fill=black]
\begin{tikzpicture}
\node[P,label=left:$$] (a)at(1,.5){};
\node[P,label=right:$$] (b)at(1.8,.5){};
\node[P,label=left:$$] (e)at(.8,1.15){};
\node[P,label=right:$$] (c)at(2,1.15){};
\node[P,label=above:$$] (d)at(1.4,1.7){};
\node()at(1.4,.36){\small{2}};
\node()at(.8,.8){\small{2}};
\node()at(2,.8){\small{2}};
\node()at(1,1.55){\small{2}};
\node()at(1.8,1.55){\small{2}};
\draw[line width=.8pt](a)--(b) --(c)--(d)--(e)--(a);
\node()at(1.4,0){$n\leq15$};
\end{tikzpicture}
\hskip.5cm
\begin{tikzpicture}
\node[P] (a)at(1,.5){};\node[P] (b)at(1.8,.5){};
\node[P] (e)at(.8,1.15){};\node[P] (c)at(2,1.15){};
\node[P] (d)at(1.4,1.7){};\node[P] (f)at(2.2,2){};
\node()at(.75,.8){\small{2}};\node()at(2.05,.8){\small{2}};
\draw[line width=.8pt](a)--(b) --(c)--(d)--(e)--(a)
(d)--(f);
\node()at(1.4,0){$n\leq14$};
\end{tikzpicture}
\hskip1cm
\begin{tikzpicture}
\node[P] (a)at(1,.5){};\node[P] (b)at(1.8,.5){};
\node[P] (e)at(.8,1.15){};\node[P] (c)at(2,1.15){};
\node[P] (d)at(1.4,1.7){};\node[P] (f)at(2.2,2){};
\draw[line width=.8pt](a)--(b) --(c)--(d)--(e)--(a)
(d)--(f)--(c);
\node()at(1.4,0){$n\leq13$};
\end{tikzpicture}
\hskip1cm
\begin{tikzpicture}
\node[P] (a)at(1,.5){};\node[P] (b)at(1.8,.5){};
\node[P] (e)at(.8,1.15){};\node[P] (c)at(2,1.15){};
\node[P] (d)at(1.4,1.7){};\node[P] (f)at(2.2,2){};
\draw[line width=.8pt](a)--(b) --(c)--(d)--(e)--(a)
(c)--(e) (d)--(f);
\node()at(1.4,0){$n\leq13$};
\end{tikzpicture}

\begin{tikzpicture}
\node[P] (a)at(1,.5){};\node[P] (b)at(1.8,.5){};
\node[P] (e)at(.8,1.15){};\node[P] (c)at(2,1.15){};
\node[P] (d)at(1.4,1.7){};\node()at(1.4,.32){\small{2}};
\node()at(.95,1.55){\small{2}};
\node()at(1.85,1.55){\small{2}};
\draw[line width=.8pt](a)--(b) --(c)--(d)--(e)--(a)
(e)--(c);
\node()at(1.4,0){$n\leq14$};
\end{tikzpicture}
\hskip1.2cm
\begin{tikzpicture}
\node[P] (a)at(1,.5){};\node[P] (b)at(1.8,.5){};
\node[P] (e)at(.8,1.15){};\node[P] (c)at(2,1.15){};
\node[P] (d)at(1.4,1.7){};
\node()at(1,1.6){\small{2}};\node()at(1.4,.32){\small{2}};
\draw[line width=.8pt](a)--(b) --(c)--(d)--(e)--(a)
(e)--(c)--(a);
\node()at(1.4,0){$n\leq14$};
\end{tikzpicture}
\hskip1.2cm
\begin{tikzpicture}
\node[P] (a)at(1,.5){};\node[P] (b)at(1.8,.5){};
\node[P] (e)at(.8,1.15){};\node[P] (c)at(2,1.15){};
\node[P] (d)at(1.4,1.7){};
\node()at(1,1.6){\small{2}};
\draw[line width=.8pt](a)--(b) --(c)--(d)--(e)--(a)
(c)--(e) (a)--(d);
\node()at(1.4,0){$n\leq13$};
\end{tikzpicture}
\hskip1.2cm
\begin{tikzpicture}
\node[P] (a)at(1,.5){};\node[P] (b)at(1.8,.5){};
\node[P] (e)at(.8,1.15){};\node[P] (c)at(2,1.15){};
\node[P] (d)at(1.4,1.7){};
\draw[line width=.8pt](a)--(b) --(c)--(d)--(e)--(a)
(c)--(e) (a)--(d)--(b);
\node()at(1.4,0){$n\leq13$};
\end{tikzpicture}
\caption{$G$ has a $5$-cycle and no $K_4$}
\label{5cycle}
\end{center}
\end{figure}
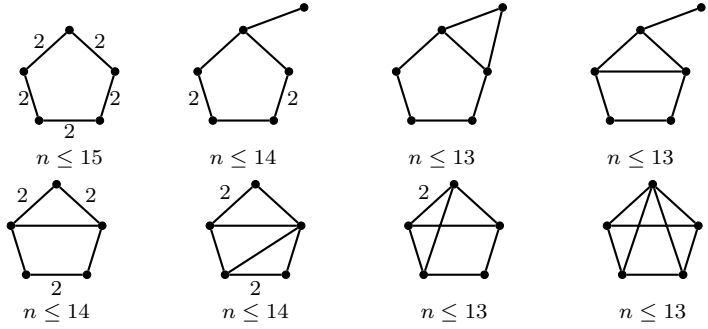

Case 2:  $G$ has a $5$-cycle and no $K_4$. Since $\tau(G)=3$, all edges of $G$ are incident with the $5$-cycle. Fig.\ref{5cycle} lists all configurations with no zero edge weights (unlabeled edges have weights $1$). All candidates have at most 
 ${4+2\choose 2}=15$ vertices.\vspace{.5em}

 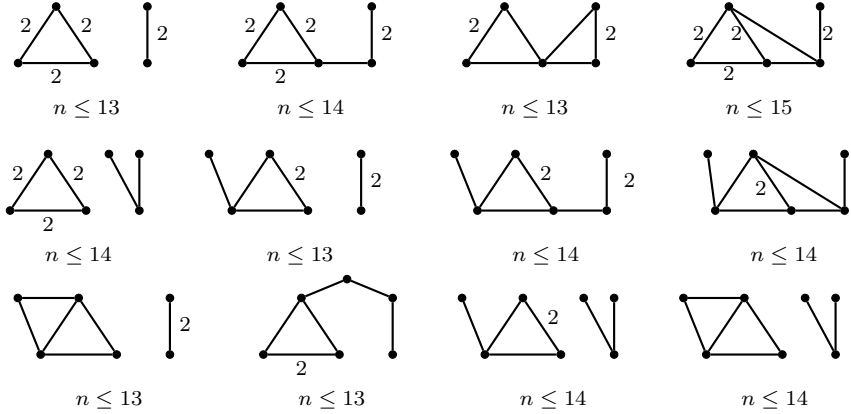
\begin{figure}[htp]
 \begin{center}
 \tikzstyle{A} = [circle, draw=black!, minimum width=1pt, fill=white,inner sep=1pt]
\tikzstyle{P} = [circle, draw=black!, minimum width=0pt, inner sep=1pt, fill=black]
\begin{tikzpicture}
\node[P] (a)at(.9,.5){};\node[P] (b)at(1.9,.5){};
\node[P] (c)at(1.4,1.25){};
\node[P] (d)at(2.6,.5){};\node[P] (e)at(2.6,1.25){};
\node()at(1.4,.32){\small{2}};
\node()at(1,1){\small{2}};
\node()at(1.8,1){\small{2}};
\node()at(2.8,.9){\small{2}};
\draw[line width=.8pt](a)--(b) --(c)--(a) (d)--(e);
\node()at(1.8,-.1){$n\leq13$};
\end{tikzpicture}
\hskip.7cm
\begin{tikzpicture}
\node[P] (a)at(.9,.5){};\node[P] (b)at(1.9,.5){};
\node[P] (c)at(1.4,1.25){};
\node[P] (d)at(2.6,.5){};\node[P] (e)at(2.6,1.25){};
\node()at(1.4,.32){\small{2}};
\node()at(1,1){\small{2}};
\node()at(1.8,1){\small{2}};
\node()at(2.8,.9){\small{2}};
\draw[line width=.8pt](a)--(b) --(c)--(a) (b)--(d)--(e);
\node()at(1.8,-.1){$n\leq14$};
\end{tikzpicture}
\hskip.7cm
\begin{tikzpicture}
\node[P] (a)at(.9,.5){};\node[P] (b)at(1.9,.5){};
\node[P] (c)at(1.4,1.25){};
\node[P] (d)at(2.6,.5){};\node[P] (e)at(2.6,1.25){};
\node()at(1,1){\small{2}};
\node()at(2.8,.9){\small{2}};
\draw[line width=.8pt](a)--(b) --(c)--(a) (b)--(d)--(e)--(b);
\node()at(1.8,-.1){$n\leq13$};
\end{tikzpicture}
\hskip.7cm
\begin{tikzpicture}
\node[P] (a)at(.9,.5){};\node[P] (b)at(1.9,.5){};
\node[P] (c)at(1.4,1.25){};
\node[P] (d)at(2.6,.5){};\node[P] (e)at(2.6,1.25){};
\node()at(1.4,.36){\small{2}};
\node()at(1,.9){\small{2}};
\node()at(1.5,.9){\small{2}};
\node()at(2.7,.9){\small{2}};
\draw[line width=.8pt](a)--(b) --(c)--(a) (b)--(d)--(e) (c)--(d);
\node()at(1.8,-.1){$n\leq15$};
\end{tikzpicture}
\vskip.3cm

\begin{tikzpicture}
\node[P] (a)at(.9,.5){};\node[P] (b)at(1.9,.5){};
\node[P] (c)at(1.4,1.25){};
\node[P] (d)at(2.6,.5){};\node[P] (e)at(2.2,1.25){};
\node[P] (f)at(2.6,1.25){};
\node()at(1.4,.32){\small{2}};
\node()at(1,1){\small{2}};
\node()at(1.8,1){\small{2}};
\draw[line width=.8pt](a)--(b) --(c)--(a) (f)--(d)--(e);
\node()at(1.8,-.1){$n\leq14$};
\end{tikzpicture}
\hskip.7cm
\begin{tikzpicture}
\node[P] (g)at(.6,1.25){};
\node[P] (a)at(.9,.5){};\node[P] (b)at(1.9,.5){};
\node[P] (c)at(1.4,1.25){};
\node[P] (d)at(2.6,.5){};\node[P] (e)at(2.6,1.25){};
\node()at(1.8,1){\small{2}};
\node()at(2.8,.9){\small{2}};
\draw[line width=.8pt](g)--(a)--(b) --(c)--(a) (d)--(e);
\node()at(1.8,-.1){$n\leq13$};
\end{tikzpicture}
\hskip.7cm
\begin{tikzpicture}
\node[P] (g)at(.6,1.25){};
\node[P] (a)at(.9,.5){};\node[P] (b)at(1.9,.5){};
\node[P] (c)at(1.4,1.25){};
\node[P] (d)at(2.6,.5){};\node[P] (e)at(2.6,1.25){};
\node()at(1.8,1){\small{2}};
\node()at(2.9,.9){\small{2}};
\draw[line width=.8pt](g)--(a)--(b) --(c)--(a) 
(d)--(e) (b)--(d);
\node()at(1.8,-.1){$n\leq14$};
\end{tikzpicture}
\hskip.7cm
\begin{tikzpicture}
\node[P] (g)at(.8,1.25){};
\node[P] (a)at(.9,.5){};\node[P] (b)at(1.9,.5){};
\node[P] (c)at(1.4,1.25){};
\node[P] (d)at(2.6,.5){};\node[P] (e)at(2.6,1.25){};
\node()at(1.5,.8){\small{2}};
\draw[line width=.8pt](g)--(a)--(b) --(c)--(a) 
(b)--(d)--(e) (c)--(d);
\node()at(1.8,-.1){$n\leq14$};
\end{tikzpicture}

\begin{tikzpicture}
\node[P] (g)at(.6,1.25){};
\node[P] (a)at(.9,.5){};\node[P] (b)at(1.9,.5){};
\node[P] (c)at(1.4,1.25){};
\node[P] (d)at(2.6,.5){};\node[P] (e)at(2.6,1.25){};
\node()at(2.8,.9){\small{2}};
\draw[line width=.8pt](a)--(b) --(c)--(a) (d)--(e)
(a)--(g)--(c);
\node()at(1.8,-.1){$n\leq13$};
\end{tikzpicture}
\hskip.7cm
\begin{tikzpicture}
\node[P] (g)at(2,1.5){};
\node[P] (a)at(.9,.5){};\node[P] (b)at(1.9,.5){};
\node[P] (c)at(1.4,1.25){};
\node[P] (d)at(2.6,.5){};\node[P] (e)at(2.6,1.25){};
\node()at(1.4,.32){\small{2}};
\draw[line width=.8pt](a)--(b) --(c)--(a) 
(d)--(e)--(g)--(c);
\node()at(1.8,-.1){$n\leq13$};
\end{tikzpicture}
\hskip.7cm
\begin{tikzpicture}
\node[P] (g)at(.6,1.25){};
\node[P] (a)at(.9,.5){};\node[P] (b)at(1.9,.5){};
\node[P] (c)at(1.4,1.25){};
\node[P] (d)at(2.6,.5){};\node[P] (e)at(2.2,1.25){};
\node[P] (f)at(2.6,1.25){};
\node()at(1.8,1){\small{2}};
\draw[line width=.8pt](g)--(a)--(b) --(c)--(a) (f)--(d)--(e);
\node()at(1.8,-.1){$n\leq14$};
\end{tikzpicture}
\hskip.7cm
\begin{tikzpicture}
\node[P] (g)at(.6,1.25){};
\node[P] (a)at(.9,.5){};\node[P] (b)at(1.9,.5){};
\node[P] (c)at(1.4,1.25){};
\node[P] (d)at(2.6,.5){};\node[P] (e)at(2.2,1.25){};
\node[P] (f)at(2.6,1.25){};
\draw[line width=.8pt](a)--(b) --(c)--(a) (f)--(d)--(e)
(a)--(g)--(c);
\node()at(1.8,-.1){$n\leq14$};
\end{tikzpicture}
\caption{$G$ has a triangle and no $5$-cycle or $K_4$}
\label{triangle}
\end{center}
\end{figure}

Case 3: $G$ has a triangle, it has no $5$-cycle and no $K_4$. The $\tau$-critical subgraph in $G$ is $K_2+K_3$. To get all candidates  at most three edges and one new vertex can be added to  $K_2+K_3$ by keeping the transversal number $3$ and creating neither a $5$-cycle nor a $K_4$. These graphs are  listed in Fig.\ref{triangle},
all have at most $15$ vertices.
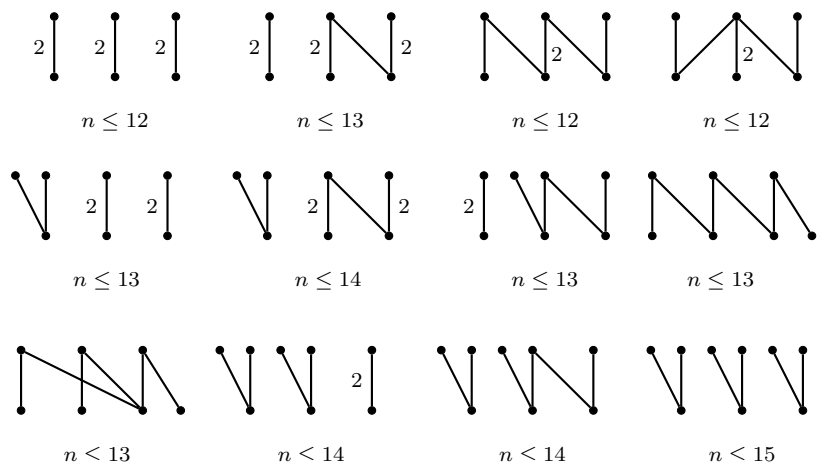
\begin{figure}[http]
 \begin{center}
 \tikzstyle{A} = [circle, draw=black!, minimum width=1pt, fill=white,inner sep=1pt]
\tikzstyle{P} = [circle, draw=black!, minimum width=0pt, inner sep=1pt, fill=black]
\begin{tikzpicture}
\node[P] (a)at(.5,.5){};\node[P] (c)at(1.3,.5){};
\node[P] (e)at(2.1,.5){};
\node[P] (b)at(.5,1.3){};\node[P] (d)at(1.3,1.3){};
\node[P] (f)at(2.1,1.3){};
\node()at(.3,.9){\small{2}};
\node()at(1.1,.9){\small{2}};
\node()at(1.9,.9){\small{2}};
\draw[line width=.8pt](a)--(b) (c)--(d) (f)--(e);
\node()at(1.3,-.1){$n\leq12$};
\end{tikzpicture}
\hskip.7cm
\begin{tikzpicture}
\node[P] (a)at(.5,.5){};\node[P] (c)at(1.3,.5){};
\node[P] (e)at(2.1,.5){};
\node[P] (b)at(.5,1.3){};\node[P] (d)at(1.3,1.3){};
\node[P] (f)at(2.1,1.3){};
\node()at(.3,.9){\small{2}};
\node()at(1.1,.9){\small{2}};
\node()at(2.3,.9){\small{2}};
\draw[line width=.8pt](a)--(b) (c)--(d) (f)--(e) (d)--(e);
\node()at(1.3,-.1){$n\leq13$};
\end{tikzpicture}
\hskip.7cm
\begin{tikzpicture}
\node[P] (a)at(.5,.5){};\node[P] (c)at(1.3,.5){};
\node[P] (e)at(2.1,.5){};
\node[P] (b)at(.5,1.3){};\node[P] (d)at(1.3,1.3){};
\node[P] (f)at(2.1,1.3){};
\node()at(1.45,.8){\small{2}};
\draw[line width=.8pt](a)--(b) --(c)--(d)--(e) (f)--(e);
\node()at(1.3,-.1){$n\leq12$};
\end{tikzpicture}
\hskip.7cm
\begin{tikzpicture}
\node[P] (a)at(.5,.5){};\node[P] (c)at(1.3,.5){};
\node[P] (e)at(2.1,.5){};
\node[P] (b)at(.5,1.3){};\node[P] (d)at(1.3,1.3){};
\node[P] (f)at(2.1,1.3){};
\node()at(1.45,.8){\small{2}};
\draw[line width=.8pt](d)--(a)--(b) (c)--(d)--(e) (f)--(e);
\node()at(1.3,-.1){$n\leq12$};
\end{tikzpicture}

\vskip.4cm
\begin{tikzpicture}
\node[P] (a)at(.5,.5){};\node[P] (c)at(1.3,.5){};
\node[P] (e)at(2.1,.5){};
\node[P] (b)at(.5,1.3){};\node[P] (d)at(1.3,1.3){};
\node[P] (f)at(2.1,1.3){};
\node[P] (b1)at(.1,1.3){};
\node()at(1.1,.9){\small{2}};
\node()at(1.9,.9){\small{2}};
\draw[line width=.8pt](b1)--(a)--(b) 
(c)--(d) (e)--(f);
\node()at(1.3,-.1){$n\leq13$};
\end{tikzpicture}
\hskip.7cm
\begin{tikzpicture}
\node[P] (a)at(.5,.5){};\node[P] (c)at(1.3,.5){};
\node[P] (e)at(2.1,.5){};
\node[P] (b)at(.5,1.3){};\node[P] (d)at(1.3,1.3){};
\node[P] (f)at(2.1,1.3){};
\node[P] (b1)at(.1,1.3){};
\node()at(1.1,.9){\small{2}};
\node()at(2.3,.9){\small{2}};
\draw[line width=.8pt](b1)--(a)--(b) (c)--(d) (f)--(e) (d)--(e);
\node()at(1.3,-.1){$n\leq14$};
\end{tikzpicture}
\hskip.4cm
\begin{tikzpicture}
\node[P] (a)at(.5,.5){};\node[P] (c)at(1.3,.5){};
\node[P] (e)at(2.1,.5){};
\node[P] (b)at(.5,1.3){};\node[P] (d)at(1.3,1.3){};
\node[P] (f)at(2.1,1.3){};
\node[P] (d1)at(.9,1.3){};
\node()at(.3,.9){\small{2}};
\draw[line width=.8pt](a)--(b) (d1)--(c)--(d) (f)--(e) (d)--(e);
\node()at(1.3,-.1){$n\leq13$};
\end{tikzpicture}
\hskip.4cm
\begin{tikzpicture}
\node[P] (a)at(.5,.5){};\node[P] (c)at(1.3,.5){};
\node[P] (e)at(2.1,.5){};\node[P] (g)at(2.6,.5){};
\node[P] (b)at(.5,1.3){};\node[P] (d)at(1.3,1.3){};
\node[P] (f)at(2.1,1.3){};
\draw[line width=.8pt](a)--(b) --(c)--(d)--(e) (g)--(f)--(e);
\node()at(1.4,-.1){$n\leq13$};
\end{tikzpicture}
\vskip.4cm
\begin{tikzpicture}
\node[P] (a)at(.5,.5){};\node[P] (c)at(1.3,.5){};
\node[P] (e)at(2.1,.5){};\node[P] (g)at(2.6,.5){};
\node[P] (b)at(.5,1.3){};\node[P] (d)at(1.3,1.3){};
\node[P] (f)at(2.1,1.3){};
\draw[line width=.8pt](a)--(b) (c)--(d)--(e) (g)--(f)--(e)--(b);
\node()at(1.5,-.1){$n\leq13$};
\end{tikzpicture}
\hskip.3cm
\begin{tikzpicture}
\node[P] (a)at(.5,.5){};\node[P] (c)at(1.3,.5){};
\node[P] (e)at(2.1,.5){};
\node[P] (b)at(.5,1.3){};\node[P] (d)at(1.3,1.3){};
\node[P] (f)at(2.1,1.3){};
\node[P] (b1)at(.1,1.3){};\node[P] (d1)at(.9,1.3){};
\node()at(1.9,.9){\small{2}};
\draw[line width=.8pt](b1)--(a)--(b) 
(d1)--(c)--(d) (e)--(f);
\node()at(1.3,-.1){$n\leq14$};
\end{tikzpicture}
\hskip.7cm
\begin{tikzpicture}
\node[P] (a)at(.5,.5){};\node[P] (c)at(1.3,.5){};
\node[P] (e)at(2.1,.5){};
\node[P] (b)at(.5,1.3){};\node[P] (d)at(1.3,1.3){};
\node[P] (f)at(2.1,1.3){};
\node[P] (b1)at(.1,1.3){};\node[P] (d1)at(.9,1.3){};;
\draw[line width=.8pt](b1)--(a)--(b) (d1)--(c)--(d) (f)--(e) (d)--(e);
\node()at(1.3,-.1){$n\leq14$};
\end{tikzpicture}
\hskip.5cm
\begin{tikzpicture}
\node[P,label=below:$$] (a)at(.5,.5){};
\node[P,label=above:$$] (a1)at(.1,1.3){};
\node[P,label=above:$$] (a2)at(.5,1.3){};
\node[P,label=below:$$] (b)at(1.3,.5){};
\node[P,label=above:$$] (b1)at(.9,1.3){};
\node[P,label=above:$$] (b2)at(1.3,1.3){};
\node[P,label=above:$$] (c2)at(2.1,1.3){};
\node[P,label=below:$$] (c)at(2.1,.5){};
\node[P,label=above:$$] (c1)at(1.7,1.3){};
\draw[line width=.8pt](a1)--(a)--(a2) 
(b1)--(b)--(b2) (c1)--(c)--(c2);
\node()at(1.3,-.1){$n\leq15$};
\end{tikzpicture}
\caption{$G$ has  no cycle}
\label{nocycle}
\end{center}
\end{figure}

\section{}
The proof of the case $m=4$ in Proposition \ref{main} shows that 
a $3$-uniform hypergraph on $n={m+2\choose 2}=15$ vertices satisfying the condition of the proposition has its private pairs graph $G$ among three candidates;  furthermore, there is only one candidate, the $5$-cycle in Case 2, that passes the Triples test. 
As a corollary, we obtain a unique $3$--uniform hypergraph $H$ of order $15$ with $\omega(H)=11$ such that its all $11$-cliques have no common 
vertex. This hypergraph is specified 
as follows.

Let $X=\{x_1,\dots,x_{10}\}$,  $Y=\{y_1,\dots,y_{5}\}$, and let $p_i=\{y_i,y_{i+1}\}$, $1\leq i\leq 5$ (where $y_6=y_1$).
For $i=1,\dots,  10$ set $N_i=(X\setminus\{x_i\})\cup p_i$. The family 
$\F=\{N_1,\dots,N_{10}\}$ defines the optimal $3$-uniform hypergraph $H$ on vertex set $X\cup Y$ with edge set including all triples contained by some $N_i$.  It is straightforward to check that the maximum cliques of $H$ are precisely the members of $\F$, and each vertex,  $x_i$ or $y_j$, is avoided by some member of $\F$.


\begin{thebibliography}{99}
\bibitem{BB} B. Bollob\'as, On generalized graphs, Acta Math. Acad. Sci. Hungar. 16 (1965)
447--452.

\bibitem{EG} P. Erd\H{o}s and T. Gallai, On the maximal number of vertices representing the edges of a graph, K\"ozl. MTA Mat. Kutat\'o Int. Budapest 6 (1961) 181--203.

\bibitem{GYLT} A. Gy\'arf\'as, J. Lehel, and  Zs. Tuza, Upper bound on the order of $\tau$-critical hypergraphs. J. Combin. Theory Ser. B 33 (1982) 161--165. 

\bibitem{GYL}
A. Gy\'arf\'as, and J. Lehel, Order plus size of $\tau$-critical graphs. J Graph
Theory. 2020;1--2.

\bibitem{JKP} A. Jobson, A. K\'ezdy, T. Pervenecki, 
On a conjecture of Szemer\'edi and Petruska. arXiv:1904.04921v2
(2019)

\bibitem{eck} A. Jobson, A. K\'ezdy, and J. Lehel,  Eckhoff's problem on planar convex sets. 2019. Note.

\bibitem{petrus} A. Jobson, A. K\'ezdy, J. Lehel, T. Pervenecki, and G. T\'oth, Petruska's question on planar convex sets. 
Discrete Math. 343 (2020) 13pp.
   
\bibitem{SzP} E. Szemer\'edi, and G. Petruska, On a combinatorial problem I. Studia Sci. Math. Hungar. 7 (1972) 363--374.

\bibitem{Tu85} Zs. Tuza, Critical hypergraphs and intersecting set-pair systems. J. Combin. Theory Ser. B 39 (1985) 134--145.

\end{thebibliography}
\end{document}